\documentclass[11pt,reqno]{amsart}
 \usepackage{amssymb,amsfonts,amscd,amsbsy, color, amsmath}
\usepackage[mathscr]{eucal}
\usepackage{url}
\usepackage{graphicx}
\usepackage{mathtools}

\newtheorem{theorem}{Theorem}[section]
\newtheorem{lemma}[theorem]{Lemma}
\newtheorem{proposition}[theorem]{Proposition}
\newtheorem{corollary}[theorem]{Corollary}
\theoremstyle{definition}

\numberwithin{equation}{section}
\newtheorem{conjecture}[theorem]{Conjecture}

\makeatletter
\def\imod#1{\allowbreak\mkern5mu({\operator@font mod}\,\,#1)}
\makeatother
\allowdisplaybreaks

\begin{document}

\title[Variations of Andrews-Beck type congruences]{Variations of Andrews-Beck type congruences}

\author{Song Heng Chan}

\author{Renrong Mao}

\author{Robert Osburn}

\address{Division of Mathematical Sciences, School of Physical and Mathematical Sciences, Nanyang Technological University, 21 Nanyang Link, Singapore 637371, Singapore}

\email{chansh@ntu.edu.sg}

\address{Department of Mathematics, Soochow University, SuZhou 215006, PR China}

\email{rrmao@suda.edu.cn}

\address{School of Mathematics and Statistics, University College Dublin, Belfield, Dublin 4, Ireland}

\email{robert.osburn@ucd.ie}


\dedicatory{In memory of Freeman Dyson}

\date{\today}

\begin{abstract}
We prove three variations of recent results due to Andrews on congruences for $NT(m,k,n)$, the total number of parts in the partitions of $n$ with rank congruent to $m$ modulo $k$. We also conjecture new congruences and relations for $NT(m,k,n)$ and for a related crank-type function.
\end{abstract}

\maketitle

\section{Introduction}

A partition $\lambda$ of a natural number $n$ is a non-increasing sequence of positive integers whose sum is $n$. For example, the $5$ partitions of $4$ are

\begin{equation*}
\begin{gathered}
4, 3+1, 2+2, 2+1+1, 1+1+1+1.
\end{gathered}
\end{equation*}
In 1944, Dyson \cite{dyson} introduced the rank of a partition as the largest part $\ell(\lambda)$ minus the number of parts $n(\lambda)$ and, based on numerical evidence, conjectured that this statistic gives a combinatorial explanation of Ramanujan's congruences for the partition function modulo $5$ and $7$. In 1954, Atkin and Swinnerton-Dyer \cite{asd} confirmed Dyson's conjecture by proving explicit formulas for the generating function of rank differences. Recently, Andrews employed these rank differences to prove some intriguing congruences (conjectured by Beck) for $NT(m,k,n)$, the total number of parts in the partitions of $n$ with rank congruent to $m$ modulo $k$. Specifically, we have (see Theorems 1 and 2 in \cite{andrews}) for all $n \in \mathbb{N}$
\begin{equation} \label{ntmod5}
NT(1, 5, 5n+i) - NT(4,5,5n+i) + 2NT(2,5,5n+i) - 2NT(3,5,5n+i) \equiv 0 \pmod{5}
\end{equation}
\noindent where $i=1$ or $4$ and
\begin{align} \label{ntmod7}
NT(1,7,7n+i) - NT(6,7,7n+i) & + NT(2,7,7n+i) - NT(5,7,7n+i) \nonumber \\
& - NT(3,7,7n+i) + NT(4,7,7n+i) \equiv 0 \pmod{7}
\end{align}
\noindent for $i=1$ or $5$.

It is now well-known that Dyson's rank is a special case of a general notion of rank which is defined on overpartition pairs \cite{bl}. Recall that an overpartition $\lambda$ of $n$ is a partition of $n$ in which the first occurrence of a number may be overlined. For example,  the $14$ overpartitions of $4$ are

\begin{equation*}
\begin{gathered}
4, \overline{4}, 3+1, \overline{3} + 1, 3 + \overline{1},
\overline{3} + \overline{1}, 2+2, \overline{2}
+ 2, 2+1+1, \overline{2} + 1 + 1, 2+ \overline{1} + 1, \\
\overline{2} + \overline{1} + 1, 1+1+1+1, \overline{1} + 1 + 1 +1.
\end{gathered}
\end{equation*}
An overpartition pair $(\lambda,\mu)$ of $n$ is a pair of overpartitions where the sum of all of the parts is $n$. We order the parts of $(\lambda,\mu)$ by requiring that for a positive integer $k$,
\begin{equation*} \label{order}
\overline{k}_{\lambda} > k_{\lambda} > \overline{k}_{\mu} >
k_{\mu},
\end{equation*}
where the subscript indicates to which of the two overpartitions
the part belongs. The rank of an overpartition pair $(\lambda,\mu)$ is
\begin{equation} \label{pairrank}
\ell((\lambda,\mu)) - n(\lambda) - \overline{n}(\mu) - \chi((\lambda,\mu)),
\end{equation}
where $\overline{n}(\cdot)$ is the number of overlined parts only and $\chi((\lambda,\mu))$ is defined to be $1$ if the largest part of $(\lambda,\mu)$ is non-overlined and in $\mu$, and $0$ otherwise. When $\mu$ is empty and $\lambda$ has no overlined parts, (\ref{pairrank}) becomes the rank of a partition. To illustrate, the rank of the overpartition pair
$((\overline{6},6,5,4,4,4,\overline{3},\overline{1}),(7,7,\overline{5},2,2,2))$ is $7 - 8 - 1 - 1 = - 3$, while the rank of the overpartition pair $((4,\overline{3},3,\overline{2},1),(4,4,4,\overline{1}))$ is $4 - 5 - 1 - 0 = -2$. In addition to recovering Dyson's rank, three other special cases of \eqref{pairrank} have turned out to be of significant interest: the rank of an overpartition \cite{Lo1}, the $M_2$-rank of a partition without repeated odd parts \cite{bg}, \cite{Lo-Os1} and the $M_2$-rank of an overpartition \cite{Lo2}. We recall these cases now. First, Dyson's rank extends in an obvious way to overpartitions. Second, the $M_2$-rank of a partition $\lambda$ without repeated odd parts is defined as

$$
\text{$M_{2}$-rank} \hspace{.025in}
(\lambda)=\Bigg\lceil{\frac{l(\lambda)}{2}}\Bigg\rceil - n(\lambda).
$$

\noindent Finally, the $M_2$-rank of an overpartition $\pi$ is given by

$$
\text{$M_2$-rank} \hspace{.025in} (\pi) = \bigg \lceil \frac{\ell(\pi)}{2}
\bigg \rceil - n(\pi) + n(\pi_o) - \chi(\pi)
$$
where $\pi_o$ is the subpartition consisting of the odd non-overlined parts and $\chi(\pi) = 1$ if the largest part of $\pi$ is odd and non-overlined and $\chi(\pi) = 0$ otherwise.

The purpose of this paper is to prove that instances of (\ref{ntmod5}) and (\ref{ntmod7}) also occur in these three additional situations. Let $\overline{NT2}(b,k,n)$ denote the total number of parts in the overpartitions of $n$ with $M_2$-rank congruent to $b$ modulo $k$. Our first result is the following.

\begin{theorem} \label{main1} For all $n \in \mathbb{N}$, we have
\begin{align}
\overline{NT2}(1,5,5n+2)& - \overline{NT2}(4,5,5n+2) \nonumber \\
& +2\overline{NT2}(2,5,5n+2)-2\overline{NT2}(3,5,5n+2)\equiv0\pmod{5}. \label{coneqm2o5}
\end{align}
\end{theorem}

If we let $\overline{NT}(b,k,n)$ be the total number of parts in the overpartitions of $n$ with rank congruent to $b$ modulo $k$, then our second result is as follows.

\begin{theorem} \label{main2} For all $n \in \mathbb{N}$, we have
\begin{align}
\overline{NT}(1,3,3n)-\overline{NT}(2,3,3n)&\equiv\overline{NT2}(1,3,3n)-\overline{NT2}(2,3,3n)\pmod{3}
\label{coneqov30}
\intertext{and}
\overline{NT}(1,3,3n+1)-\overline{NT}(2,3,3n+1)&\equiv\overline{NT2}(1,3,3n+1)-\overline{NT2}(2,3,3n+1) \pmod{3}.
\label{coneqov31}
\end{align}
\end{theorem}

Finally, if $NT2(b,k,n)$ is the total number of parts in the partitions of $n$ without repeated odd parts with $M_2$-rank congruent to $b$ modulo $k$, then our third result is the following.

\begin{theorem} \label{main3} For all $n \in \mathbb{N}$, we have
\begin{align}
NT2(1,5,5n+1)& - NT2(4,5,5n+1) \nonumber \\
& +2NT2(2,5,5n+1)-2NT2(3,5,5n+1)\equiv0\pmod{5}. \label{coneqm25}
\end{align}
\end{theorem}

The paper is organized as follows. In Section 2, we establish the generating function for the rank of an overpartition pair which also keeps track of the total number of parts. Upon appropriate specializations, this result leads to the generating functions for $\overline{NT2}(b,k,n)$, $\overline{NT}(b,k,n)$ and $NT2(b,k,n)$. We also record a key result necessary for the proof of Theorem \ref{main2}. In Section 3, we prove Theorems \ref{main1}--\ref{main3}. In Section 4, we make some concluding remarks concerning future directions.

\section{Preliminaries}

We first recall the standard $q$-hypergeometric notation

\begin{align}
(a)_{n} & = (a;q)_{n} := \prod_{k=1}^{n} (1-aq^{k-1}), \nonumber \\
(a_1, \ldots, a_m)_{n} & = (a_1, \ldots, a_m; q)_{n} := (a_1)_{n}, \cdots (a_m)_{n} \nonumber
 \intertext{and }\nonumber \\
[a_1, \ldots, a_m]_{n} & = [a_1, \ldots, a_m; q]_{n} = (a_1, q/a_1, \ldots, a_m, q/a_m)_{n}, \nonumber
\end{align}

\noindent valid for $n \in \mathbb{N} \cup \{\infty\}$.

Let $N(r,s,t,m,n)$ be the number of overpartition pairs $(\lambda, \mu)$ of $n$ with rank $m$, such that $r$ is the number of overlined parts in $\lambda$ plus the number of non-overlined parts in $\mu$, $s$ is the number of parts in $\mu$ and $t$ is the total number of parts in $(\lambda,\mu)$. 	

\begin{lemma}
We have
\begin{align}\label{genovpair}	
\mathcal{N}(d, e, x, z ; q):=\sum_{r, s, t, n \geq 0 \atop m \in \mathbb{Z}} N(r, s,t, m, n) d^{r} e^{s}x^t z^{m} q^{n}=\sum_{n \geq 0} \frac{(-1 / d,-1 / e)_{n}(xd e q)^{n}}{(z q, xq / z)_{n}}.
\end{align}
\end{lemma}
\begin{proof}
We follow the proof of Proposition 2.1 in \cite{bl}. We split the overpartition pairs into four
cases, depending on whether the largest part is overlined or not and
whether it is in $\lambda$ or $\mu$ to get four series. For
example, the series
\begin{align*}
	\sum_{n \geq 1} \frac{exq^{n}z^{n-1}(-qex/z, -qdx/z)_{n-1}}{(xq/z)_{n-1}(xedq)_n}
\end{align*}
is the generating function for overpartition pairs  whose largest part $n$ is in $\mu$ and overlined, where the
exponent of $q$ is the number being partitioned, the exponent
of $z$ is the rank, the exponent of $d$ is the number of overlined parts in $\lambda$ plus the number of non-overlined parts in $\mu$, the exponent of $e$ is the number of parts in $\mu$ and the exponent of $x$ is the total number of parts in $(\lambda,\mu)$. Combining this with the other three cases, we obtain
\begin{align}	
&\sum_{r, s, t, n \geq 0 \atop m \in \mathbb{Z}} N(r, s,t, m, n) d^{r} e^{s}x^t z^{m} q^{n}\nonumber\\
&=1+\sum_{n \geq 1} \frac{exq^{n}z^{n-1}(-qex/z, -qdx/z)_{n-1}}{(xq/z)_{n-1}(xedq)_n}+\sum_{n \geq 1} \frac{dexq^{n}z^{n-1}(-qex/z, -qdx/z)_{n-1}}{(xq/z)_{n-1}(xedq)_n}
\nonumber\\
&\qquad+\sum_{n \geq 1} \frac{dxq^{n}z^{n-1}(-qex/z)_{n}(-qdx/z)_{n-1}}{(xq/z, xedq)_n}+\sum_{n \geq 1} \frac{xq^{n}z^{n-1}(-qex/z)_{n}(-qdx/z)_{n-1}}{(xq/z, xedq)_n}
\nonumber\\
&=1+(x+dx)\left(\sum_{n \geq 1} \frac{eq^{n}z^{n-1}(-qex/z, -qdx/z)_{n-1}}{(xq/z)_{n-1}(xedq)_n}
+\sum_{n \geq 1} \frac{q^{n}z^{n-1}(-qex/z)_{n}(-qdx/z)_{n-1}}{(xq/z, xedq)_n}\right)
\nonumber\\
&=1+(x+dx)\sum_{n \geq 1} \frac{q^{n}z^{n-1}(-qex/z, -qdx/z)_{n-1}}{(xq/z)_{n-1}(xedq)_n}\left(e+\frac{1+q^nex/z}{1-q^nx/z}
\right)
\nonumber\\
&=1+(x+dx)(1+e)\sum_{n \geq 1} \frac{q^{n}z^{n-1}(-qex/z, -qdx/z)_{n-1}}{(xq/z, xedq)_n}
\nonumber\\
&=1+(x+dx)(1+e)\sum_{n \geq 0} \frac{q^{n+1}z^{n}(-qex/z, -qdx/z)_{n}}{(xq/z, xedq)_{n+1}}
\nonumber\\
&=1+\frac{(x+dx)(1+e)q}{(1-xq/z)(1-xedq)}\sum_{n \geq 0} \frac{q^{n}z^{n}(-qex/z, -qdx/z)_{n}}{(xq^2/z, xedq^2)_{n}}. \label{pgeovpair}
\end{align}
Replacing $(a, b, c, d, e)$ by $(q, -qdx/z, -qex/z, q^2x/z, deq^2x)$ in \cite[Eq. (3.27)]{gr}, we find that
 \begin{align}\label{fi32}
\sum_{n \geq 0} \frac{q^{n}z^{n}(-qex/z, -qdx/z)_{n}}{(xq^2/z, xedq^2)_{n}}=\frac{(dexq,q^2z)_\infty}{(dexq^2,qz)_\infty}
\sum_{n \geq 0} \frac{(dexq)^{n}(-q/d, -q/e)_{n}}{(xq^2/z, q^2z)_{n}}.
\end{align}
Substituting \eqref{fi32} into \eqref{pgeovpair}, we obtain
\begin{align*}	
&\sum_{r, s, t, n \geq 0 \atop m \in \mathbb{Z}} N(r, s,t, m, n) d^{r} e^{s}x^t z^{m} q^{n}\nonumber\\
&=1+\frac{(x+dx)(1+e)q}{(1-xq/z)(1-xedq)}\frac{(dexq,q^2z)_\infty}{(dexq^2,qz)_\infty}
\sum_{n \geq 0} \frac{(dexq)^{n}(-q/d, -q/e)_{n}}{(xq^2/z, q^2z)_{n}}
\nonumber\\&=1+\frac{(x+dx)(1+e)q}{dexq(1+1/d)(1+1/e)}
\sum_{n \geq 0} \frac{(dexq)^{n+1}(-1/d, -1/e)_{n+1}}{(xq/z, qz)_{n+1}}
\nonumber\\&=
\sum_{n \geq 0} \frac{(dexq)^{n}(-1/d, -1/e)_{n}}{(xq/z, qz)_{n}}.
\end{align*}
\end{proof}

Next, we prove the following result which generalizes \cite[Theorem 3]{andrews}.
\begin{proposition}\label{main}
  We have
 \begin{align}\label{thmain}
&\sum_{n \geq 0} \frac{(-1/d,-1/e)_n (xdeq)^{n} }{(zq, xq/z)_n}
\nonumber\\&= 1 - \frac{(-xqd,-xqe)_{\infty}}{(xq,xqde)_{\infty}}\sum_{n \geq 1}\frac{ (xq,-1/d,-1/e)_n q^{n(n+3)/2} (-dex)^n}{(q)_{n-1} (-xdq,-xeq)_n} \left(\frac{1}{q^{n}(1-zq^n)}+\frac{xz^{-1}}{1 - xq^n/z }\right).
\end{align}
\end{proposition}
\begin{proof}
Recall the limiting case of \cite[Eq. (2.5.1)]{gr}:
 \begin{align}
&\sum_{n \geq 0} \frac{(aq/bc,d,e)_n (aq/de)^{n}}{(q,aq/b,aq/c)_n }
\nonumber\\&= \frac{(aq/d,aq/e)_{\infty}}{(aq,aq/de)_{\infty}}\sum_{n \geq 0}\frac{ (a,\sqrt{a}q,-\sqrt{a}q,b,c,d,e)_n
(-a^2q^2)^nq^{n(n-1)/2}}{ (q,\sqrt{a},-\sqrt{a},aq/b,aq/c,aq/d,aq/e)_n(bcde)^n}.
\label{watson}
\end{align}
Replacing $(a,b,c,d,e)$ by $(x,x/z,z,-1/d,-1/e)$ in \eqref{watson}, we have
 \begin{align*}
&\sum_{n \geq 0} \frac{(-1/d,-1/e)_n (xdeq)^{n} }{(zq, xq/z)_n}
\nonumber\\&= \frac{(-xqd,-xqe)_{\infty}}{(xq,xqde)_{\infty}}\sum_{n \geq 0}\frac{ (x,\sqrt{x}q,-\sqrt{x}q,x/z,z,-1/d,-1/e)_n
(-xdeq)^nq^{n(n+1)/2}}{ (q,\sqrt{x},-\sqrt{x},qz,xq/z,-xqd,-xqe)_n}.
\end{align*}
After simplifying, we find that
 \begin{align}\label{pthmain1}
&\sum_{n \geq 0} \frac{(-1/d,-1/e)_n (xdeq)^{n} }{(zq, xq/z)_n}
\nonumber\\&= \frac{(-xqd,-xqe)_{\infty}}{(xq,xqde)_{\infty}}\left\{1+\sum_{n \geq 1}\frac{(xq)_{n-1}\left(-1/d, -1/e\right)_n (-xde)^nq^{n(n+3)/2}(1-z)(1- x/z)(1-xq^{2n})}{(q,-xdq,-xeq)_n(1-zq^n)(1- xq^n/z)} \right\}.
\end{align}
Noting that
 \begin{align*}
&\frac{(1-z)(1-x/z)(1-xq^{2n})}{(1-zq^n)(1-xq^n/z)}
\nonumber\\&=-(1-q^n)(1-xq^n)\left(\frac{1}{q^{n}(1-zq^n)}+\frac{xz^{-1}}{1 - xq^n/z }\right)+\frac{1-xq^{2n}}{q^n},
\end{align*}
we deduce from \eqref{pthmain1} that
 \begin{align}\label{pthmain2}
&\sum_{n \geq 0} \frac{(-1/d,-1/e)_n (xdeq)^{n} }{(zq, xq/z)_n}
\nonumber\\
&= -\frac{(-xqd,-xqe)_{\infty}}{(xq,xqde)_{\infty}}\sum_{n \geq 1}\frac{ \left(xq,-1/d,-1/e \right)_n (-xde)^nq^{n(n+3)/2}}{(q)_{n-1}(-xdq,-xeq)_n} \left(\frac{1}{q^{n}(1-zq^n)}+\frac{xz^{-1}}{1- xq^n/z} \right)\nonumber\\
&\quad + \frac{(-xqd,-xqe)_{\infty}}{(xq,xqde)_{\infty}}\left\{1+\sum_{n \geq 1}\frac{ (xq)_{n-1}\left(-1/d,-1/e \right)_n (-xde)^n q^{n(n+1)/2}(1-xq^{2n})}{(q,-xdq,-xeq)_n} \right\}.
\end{align}
Replacing $(a,b,c,d)$ by $(x,-1/d,-1/e,\infty)$ in \cite[Eq. (II.20), p.356]{gr}, we obtain
 \begin{align*}
1+
\sum_{n \geq 1}\frac{ (xq)_{n-1}\left(-1/d, -1/e \right)_n (-xde)^n q^{n(n+1)/2} (1-xq^{2n})}{(q,-xdq,-xeq)_n} =
\frac{(xq,xqde)_{\infty}}{(-xqd,-xqe)_{\infty}},
\end{align*}
which together with \eqref{pthmain2} gives \eqref{thmain}.
\end{proof}

Let $\overline{NNT}(r,s,b,k,n)$ denote the total number of parts in the overpartition pairs $(\lambda, \mu)$ of $n$ with rank congruent to $b$ modulo $k$, such that
$r$ is the number of overlined parts in $\lambda$ plus the number of non-overlined parts in $\mu$, $s$ is the number of parts in $\mu$. Proceeding as in the proof of \cite[Corollary 4]{andrews}, one can obtain the following generalization of \cite[Corollary 4]{andrews} by applying Proposition \ref{main}.

\begin{corollary}\label{legennntkb}
For $1\leq b\leq k-1$, we have
\begin{align*}
&\sum_{r, s, n \geq 0 }\left (\overline{NNT}(r, s,b, k, n) -\overline{NNT}(r, s,k-b, k, n)\right)d^{r} e^{s} q^{n}
 \nonumber\\
& = - \left. \frac{\partial}{\partial x} \right |_{x=1} \frac{(-xqd,-xqe)_{\infty}}{(xq,xqde)_{\infty}}\sum_{n \geq 1}\frac{ (xq,-1/d,-1/e)_n q^{n(n+3)/2} (-dex)^n}{(q)_{n-1} (-xdq,-xeq)_n}
 \nonumber\\&\quad\cdot \left(\frac{ q^{(b-1)n}-q^{n(k-b-1)}}{1 - q^{kn}}+\frac{x^{k-b} q^{(k-1-b)n}-x^{b} q^{(b-1)n}}{1 - x^k q^{kn}}\right).
\end{align*}
\end{corollary}

Finally, we prove a key Lemma required in the proof of Theorem \ref{main2}.

\begin{lemma} \label{l42} We have
  \begin{align*}  \nonumber
\frac{[q^3;q^9]_\infty^3 (q^9;q^9)_\infty^2}{[-q^3;q^9]_\infty^2(-q^9;q^9)_\infty^2}
& =
2\sum_{n \in \mathbb{Z}} \frac{(-1)^n q^{9n^2+6n}}{1+q^{9n}}
-2\sum_{n \in \mathbb{Z}} \frac{(-1)^n q^{9n^2+12n+3}}{1+q^{9n+3}} \\
&+4\frac{(-q^9;q^9)_\infty^2} {[-q^{3};q^9]_\infty}
\sum_{n \in \mathbb{Z}} \frac{(-1)^n q^{9n^2+18n+9}}{1+q^{9n+6}} .
\end{align*}
\end{lemma}

\begin{proof}
Setting  $r=1$ and $s=3$ in \cite[Theorem 2.1]{gen}, we see that
\begin{align*}  \nonumber
\frac{[a]_\infty (q)_\infty^2}{[b_1, b_2, b_3]_\infty}
& =
 \frac{[a/b_1]_\infty} {[b_2/b_1, b_3/b_1 ]_\infty}
\sum_{n \in \mathbb{Z}} \frac{q^{n(n+1)}}{1-b_1q^n} \left( \frac{a b_1 }{b_2  b_3} \right)^n\\
&+ \frac{[a/b_2]_\infty} {[b_1/b_2, b_3/b_2 ]_\infty}
\sum_{n \in \mathbb{Z}} \frac{q^{n(n+1)}}{1-b_2q^n} \left( \frac{a b_2}{b_1  b_3} \right)^n\\
& + \frac{[a/b_3]_\infty} {[b_1/b_3, b_2/b_3 ]_\infty}
\sum_{n \in \mathbb{Z}} \frac{q^{n(n+1)}}{1-b_3q^n} \left( \frac{a b_3 }{b_1  b_2} \right)^n.
\end{align*}
Replacing $q$ by $q^9$ and setting $(a, b_1, b_2, b_3)=(q^6, -1, -q^3,-q^6)$, we find that
\begin{align*}  \nonumber
\frac{[q^6;q^9]_\infty (q^9;q^9)_\infty^2}{[-1, -q^3, -q^6;q^9]_\infty}
& = \frac{[-q^6;q^9]_\infty} {[q^3,q^6;q^9 ]_\infty}
\sum_{n \in \mathbb{Z}} \frac{(-1)^n q^{9n^2+6n}}{1+q^{9n}} \\
&+ \frac{[-q^3;q^9]_\infty} {[q^{-3}, q^3;q^9 ]_\infty}
\sum_{n \in \mathbb{Z}} \frac{(-1)^n q^{9n^2+12n}}{1+q^{9n+3}} \\
&+\frac{[-1;q^9]_\infty} {[q^{-3},q^{-6};q^9]_\infty}
\sum_{n \in \mathbb{Z}} \frac{(-1)^n q^{9n^2+18n}}{1+q^{9n+6}} .
\end{align*}
Multiplying both sides by $2[q^3;q^9]_\infty^2/[-q^3;q^9]_\infty$ and simplifying completes the proof.
\end{proof}

\section{Proofs of Theorems \ref{main1}--\ref{main3}}

We are now in a position to prove Theorems \ref{main1}--\ref{main3}.

\begin{proof}[Proof of Theorem \ref{main1}]
We first consider the following special case of \eqref{genovpair}:
\begin{align*}	
\mathcal{N}(1, 1/q, x, z ; q^2)=\sum_{n \geq 0} \frac{(-1,-q;q^2)_n (xq)^{n} }{(zq,xq/z;q^2)_n}.
\end{align*}
By \cite[Theorem 1.2]{Lo2}, the coefficient of $x^tz^bq^n$ in
\begin{align*}
\sum_{n \geq 0} \frac{(-1,-q;q^2)_n (xq)^{n} }{(zq,xq/z;q^2)_n}
\end{align*}
is equal to the number of overpartitions of $n$ with $t$ parts and $M_2$-rank $b$. By Corollary \ref{legennntkb}, we have for $1\leq b\leq k-1$
\begin{align}\label{genovm2nt}
&\sum_{n \geq 0} \left(\overline{NT2}(b,k,n)-\overline{NT2}(k-b,k,n)\right) q^n
 \nonumber\\
 &= - \left. \frac{\partial}{\partial x} \right |_{x=1} \frac{(-xq)_{\infty}}{(xq)_{\infty}}\sum_{n \geq 1}\frac{ (xq^2,-1,-q;q^2)_n q^{n(n+2)} (-x)^n}{(q^2;q^2)_{n-1} (-xq^2,-xq;q^2)_n}
 \nonumber\\
&\quad\cdot \left(\frac{ q^{2(b-1)n}-q^{2(k-b-1)n}}{1 - q^{2kn}}+\frac{x^{k-b} q^{2(k-1-b)n}-x^{b} q^{2(b-1)n}}{1 - x^k q^{2kn}}\right).
\end{align}
Using \eqref{genovm2nt} with $b=1$ and $k=5$, we obtain
\begin{align}
&\sum_{n \geq 0} \left(\overline{NT2}(1,5,n)-\overline{NT2}(4,5,n)\right) q^n
 \nonumber\\
 &= -\left. \frac{\partial}{\partial x} \right |_{x=1} \frac{(-xq)_{\infty}}{(xq)_{\infty}}\sum_{n \geq 1}\frac{ (xq^2,-1,-q;q^2)_n q^{n(n+2)} (-x)^n}{(q^2;q^2)_{n-1} (-xq^2,-xq;q^2)_n}\left(\frac{ 1-q^{6n}}{1 - q^{10n}}+\frac{x^{4} q^{6n}-x}{1 - x^5 q^{10n}}\right) \nonumber\\
 &= - \left. \frac{\partial}{\partial x} \right |_{x=1} \frac{(-xq)_{\infty}}{(xq)_{\infty}}\sum_{n \geq 1}\frac{ (xq^2,-1,-q;q^2)_n q^{n(n+2)} (-x)^n}{(q^2;q^2)_{n-1} (-xq^2,-xq;q^2)_n}
 \nonumber\\
 &\quad\cdot \frac{(x - 1) (q^{2n} - 1) (x q^{2n} - 1) (x q^{4 n} - 1)}{(1 - q^{10n})(1 - x^5 q^{10n})} \nonumber\\&\quad
 \cdot\{1 + q^{2n} + q^{4 n} + q^{2n} x + 2 q^{4 n} x + q^{6 n} x + q^{4n} x^2 + q^{6 n} x^2 + q^{8 n} x^2\}. \nonumber
\end{align}
Noting that
\begin{align}
\left.\frac{\partial}{\partial x}\right|_{x=1}(1-x) F(x, q, z)=-F(1, q, z),\label{partial}
\end{align}
we find
\begin{align}
&\sum_{n \geq 0} \left(\overline{NT2}(1,5,n)-\overline{NT2}(4,5,n)\right) q^n
 \nonumber\\
&=  \frac{2(-q)_{\infty}}{(q)_{\infty}}\sum_{n \geq 1}
 \frac{ (-1)^{n}q^{n(n+2)} (q^{2n} - 1)^3 ( q^{4 n} - 1)}{(1+q^{2n})(1 - q^{10n})^2} \{1 + 2q^{2n} + 4q^{4 n} +  2q^{6 n}  +  q^{8 n} \}.\label{conm2ovmod51}
\end{align}
Similarly, one can prove that
\begin{align}
&\sum_{n \geq 0} \left(\overline{NT2}(2,5,n)-\overline{NT2}(3,5,n)\right) q^n
 \nonumber\\&=  \frac{2(-q)_{\infty}}{(q)_{\infty}}\sum_{n \geq 1}
 \frac{(-1)^{n} q^{n(n+2)} (q^{2n} - 1)^3 ( q^{4 n} - 1)}{(1+q^{2n})(1 - q^{10n})^2} \{2q^{2n} + q^{4 n} +  2q^{6 n}   \}.\label{conm2ovmod52}
\end{align}
Equations \eqref{conm2ovmod51} and \eqref{conm2ovmod52} give
\begin{align}
&\sum_{n \geq 0} \left(\overline{NT2}(1,5,n)-\overline{NT2}(4,5,n)+2\overline{NT2}(2,5,n)-2\overline{NT2}(3,5,n)\right) q^n
 \nonumber\\
 &=  \frac{2(-q)_{\infty}}{(q)_{\infty}}\sum_{n \geq 1}
 \frac{(-1)^{n} q^{n(n+2)} (q^{2n} - 1)^3 ( q^{4 n} - 1)}{(1+q^{2n})(1 - q^{10n})^2} \{1 + 6q^{2n} + 6q^{4 n} +  6q^{6 n}  +  q^{8 n}  \} \nonumber\\
 &\equiv  \frac{2(-q)_{\infty}}{(q)_{\infty}}\sum_{n \geq 1}
 \frac{ (-1)^{n}q^{n(n+2)} (q^{2n} - 1)^2 ( 1-q^{4 n} )}{(1+q^{2n})(1 - q^{10n})}\pmod{5}
  \nonumber\\
  &=  \frac{2(-q)_{\infty}}{(q)_{\infty}}\sum_{n \geq 1}
 \frac{ (-1)^{n}q^{n(n+2)} (1-q^{2n} )^3}{1 - q^{10n}}
    \nonumber\\&=  \frac{2(-q)_{\infty}}{(q)_{\infty}}\sum_{{n \in \mathbb{Z} \atop n\neq0}}
 \frac{ (-1)^{n}q^{n(n+2)} (1-3q^{2n} )}{1 - q^{10n}}
 \nonumber\\&=  \frac{2(-q)_{\infty}}{(q)_{\infty}}\left(\bar{S}_{2}(1)+3\bar{S}_{2}(3)\right), \label{ovs131}
 \intertext{where}
 \bar{S}_{2}(b):&=\sum_{{n \in \mathbb{Z} \atop n \neq0}} \frac{(-1)^{n} q^{n^{2}+2 b n}}{1-q^{10 n}}. \nonumber
\end{align}
By \cite[Eq. (4.9)]{Lo-Os2}, we have
\begin{equation}\label{ovs132}
\sum_{n \geq 0} \Bigl \{ \overline{N}_{2}(1,5, n)-\overline{N}_{2}(2,5, n) \Bigr \} q^{n} \frac{(q)_{\infty}}{2(-q)_{\infty}}=-\bar{S}_{2}(1)-3 \bar{S}_{2}(3),
\end{equation}
where $\overline{N}_{2}(b,k, n)$ denotes the number of overpartitions of $n$ whose $M_2$-rank is congruent to $b$ modulo $k$.
Equations \eqref{ovs131} and \eqref{ovs132} imply
\begin{align*}
&\sum_{n \geq 0} \left(\overline{NT2}(1,5,n)-\overline{NT2}(4,5,n)+2\overline{NT2}(2,5,n)-2\overline{NT2}(3,5,n)\right) q^n
\nonumber\\&\equiv\sum_{n \geq 0} \left\{\overline{N}_{2}(2,5, n)-\overline{N}_{2}(1,5, n)\right\} q^{n} \pmod{5},
\end{align*}
which together with \cite[Eq. (1.10)]{Lo-Os2} yields (\ref{coneqm2o5}).
\end{proof}

\begin{proof}[Proof of Theorem \ref{main2}]
Setting $e=0$ and $d=1$ in \eqref{genovpair}, we obtain
\begin{align*}
\mathcal{N}(1, 0, x, z ; q^2)=\sum_{n \geq 0} \frac{(-1)_n x^nq^{n(n+1)/2} }{(zq,xq/z)_n}.
\end{align*}
By \cite[Proposition 1.1]{Lo1}, the coefficient of $x^tz^bq^n$ in
\begin{align*}
\sum_{n \geq 0} \frac{(-1)_n x^nq^{n(n+1)/2} }{(zq,xq/z)_n}
\end{align*}
is equal to the number of overpartitions of $n$ with $t$ parts and rank $b$. By Corollary \ref{legennntkb}, we have
for $1\leq b\leq k-1$
\begin{align}\label{genovdrnt}
&\sum_{n \geq 0} \left(\overline{NT}(b,k,n)-\overline{NT}(k-b,k,n)\right) q^n
 \nonumber\\
&= - \left. \frac{\partial}{\partial x} \right |_{x=1} \frac{(-xq)_{\infty}}{(xq)_{\infty}}\sum_{n \geq 1}\frac{ (xq,-1;q)_n q^{n(n+1)} (-x)^n}{(q;q)_{n-1} (-xq;q)_n}
 \nonumber\\&\quad\cdot \left(\frac{ q^{(b-1)n}-q^{(k-b-1)n}}{1 - q^{kn}}+\frac{x^{k-b} q^{(k-1-b)n}-x^{b} q^{(b-1)n}}{1 - x^k q^{kn}}\right).
\end{align}
Proceeding as in the proof of Theorem \ref{main1},
after applying \eqref{genovm2nt}, \eqref{partial} and \eqref{genovdrnt}, we obtain
\begin{align*}
\sum_{n \geq 0} \left(\overline{NT}(1,3,n)-\overline{NT}(2,3,n)\right) q^n &=  \frac{2(-q)_{\infty}}{(q)_{\infty}}\sum_{n \geq 1}
 \frac{ (-1)^{n}q^{n^2+n} (q^{n} - 1)^4}{(1 - q^{3n})^2}
 \nonumber\\&\equiv  \frac{2(-q)_{\infty}}{(q)_{\infty}}\sum_{n \geq 1}
 \frac{ (-1)^{n}q^{n^2+n} ( 1-q^n)}{1 - q^{3n}}\pmod{3}
\intertext{and}
\sum_{n \geq 0} \left(\overline{NT2}(1,3,n)-\overline{NT2}(2,3,n)\right) q^n &=  \frac{2(-q)_{\infty}}{(q)_{\infty}}\sum_{n \geq 1}
 \frac{ (-1)^{n}q^{n^2+2n} (q^{2n} - 1)^4 }{(1 - q^{6n})^2}  \nonumber\\
 &\equiv  \frac{2(-q)_{\infty}}{(q)_{\infty}}\sum_{n \geq 1}
 \frac{ (-1)^{n}q^{n^2+2n} ( 1-q^{2n})}{1 - q^{6n}}\pmod{3}.
\end{align*}
Thus, we have
\begin{align}
 &\sum_{n \geq 0} \left(\overline{NT}(1,3,n)-\overline{NT}(2,3,n)-\overline{NT2}(1,3,n)+\overline{NT2}(2,3,n)\right) q^n
 \nonumber\\&\equiv   \frac{2(-q)_{\infty}}{(q)_{\infty}}\sum_{n \geq 1}
 \frac{ (-1)^{n}q^{n^2+n} ( 1-2q^{n}+2q^{3n}-q^{4n})}{1 - q^{6n}}
  \nonumber\\&\equiv   \frac{2(-q)_{\infty}}{(q)_{\infty}}\sum_{n \geq 1}
 \frac{ (-1)^{n}q^{n^2+n} ( 1+q^{n})}{1 + q^{3n}}\pmod{3}.  \label{dis1}
\end{align}
Now, from \cite[Eq. (2.1)]{ramphi}, we have
\begin{align} \nonumber
\frac{(-q)_\infty}{(q)_\infty}
&=
\frac{(q^{18};q^{18})_\infty^3}{[q^3;q^{18}]_\infty^8(q^6;q^6)_\infty^4[q^9;q^{18}]_\infty}
\left(1+ 2q\frac{[q^3;q^{18}]_\infty}{[q^9;q^{18}]_\infty}+4q^2\frac{[q^3;q^{18}]_\infty^2}{[q^9;q^{18}]_\infty^2}\right)\\
&= \frac{(q^{18};q^{18})_\infty^3}{[q^3;q^{18}]_\infty^8(q^6;q^6)_\infty^4[q^9;q^{18}]_\infty}
\left(1+ 2q\frac{(-q^9;q^9)_\infty^2}{[-q^3;q^{9}]_\infty}+4q^2\frac{(-q^9;q^9)_\infty^4}{[-q^3;q^9]_\infty^2}\right).
\label{dis2}
\end{align}
Next, note that
\begin{align} \nonumber
  &\sum_{n \geq 1}
 \frac{ (-1)^{n}q^{n^2+n} ( 1+q^{n})}{1 + q^{3n}}
 = -\frac{1}{2}+ \sum_{n \in \mathbb{Z}}
 \frac{ (-1)^{n}q^{n^2+n} }{1 + q^{3n}} \\
 &\quad= -\frac{1}{2}
 +
  \sum_{n \in \mathbb{Z}} \frac{ (-1)^{n}q^{9n^2+6n} }{1 + q^{9n}}
 -
 \sum_{n \in \mathbb{Z}} \frac{ (-1)^{n}q^{9n^2+12n+3} }{1 + q^{9n+3}}
 +
  \sum_{n \in \mathbb{Z}} \frac{ (-1)^{n}q^{9n^2+18n+8} }{1 + q^{9n+6}}. \label{dis3}
\end{align}
Invoking \eqref{dis2} and \eqref{dis3} into \eqref{dis1} and collecting only terms where the power $q$ is divisible by 3 yields
\begin{align*}
&\frac{(q^{18};q^{18})_\infty^3}{[q^3;q^{18}]_\infty^8(q^6;q^6)_\infty^4[q^9;q^{18}]_\infty}
\Bigg(-1
 +
2  \sum_{n \in \mathbb{Z}} \frac{ (-1)^{n}q^{9n^2+6n} }{1 + q^{9n}}
 \\
 &\quad-
 2\sum_{n \in \mathbb{Z}} \frac{ (-1)^{n}q^{9n^2+12n+3} }{1 + q^{9n+3}}
 +
 4\frac{(-q^9;q^9)_\infty^2}{[-q^3;q^{9}]_\infty}
  \sum_{n \in \mathbb{Z}} \frac{ (-1)^{n}q^{9n^2+18n+9} }{1 + q^{9n+6}}\Bigg).
\end{align*}
Applying Lemma \ref{l42}, the expression in parenthesis then becomes
\begin{align*}
-1+ \frac{[q^3;q^9]_\infty^3 (q^9;q^9)_\infty^2}{[-q^3;q^9]_\infty^3(-q^9;q^9)_\infty^2}
&=-1+\frac{(q^3;q^3)_\infty^3 (-q^9;q^9)_\infty}{(-q^3;q^3)_\infty^3(q^9;q^9)_\infty}\\
&\equiv -1+1 = 0 \pmod{3}.
\end{align*}
This shows that the coefficients of $q^{3n}$ in \eqref{dis1} are divisible by 3, which then implies \eqref{coneqov30}.
Similarly, congruence \eqref{coneqov31} can be proved in exactly the same way, by showing that the coefficients of $q^{3n+1}$ in \eqref{dis1} are also divisible by 3.
\end{proof}

\begin{proof}[Proof of Theorem \ref{main3}]
Replacing $(q, d, e)$ by $(q^2, 0, 1/q)$ in \eqref{genovpair}, we obtain
\begin{align*}
\mathcal{N}(0, 1/q, x, z ; q^2)=\sum_{n \geq 0} \frac{(-q;q^2)_n x^nq^{n^2} }{(zq,xq/z;q^2)_n}.
\end{align*}
Noting that partitions without repeated odd parts correspond to overpartitions in which the odd parts are all overlined and even parts are all non-overlined, we deduce from \cite[Theorem 1.2]{Lo2} that the coefficient of $x^tz^bq^n$ in
\begin{align*}
&\sum_{n \geq 0} \frac{(-q;q^2)_n x^nq^{n^2} }{(zq,xq/z;q^2)_n}
\end{align*}
is equal to the number of partitions without repeated odd parts of $n$ with $t$ parts and $M_2$-rank
$b$. By Corollary \ref{legennntkb}, we have for $1\leq b\leq k-1$
\begin{align}\label{genm2nt}
&\sum_{n \geq 0} \left(NT2(b,k,n)-NT2(k-b,k,n)\right) q^n
 \nonumber\\&= - \left. \frac{\partial}{\partial x} \right |_{x=1} \frac{(-xq;q^2)_{\infty}}{(xq^2;q^2)_{\infty}}\sum_{n \geq 1}\frac{ (xq^2,-q;q^2)_n q^{2n^2+n} (-x)^n}{(q^2;q^2)_{n-1} (-xq;q^2)_n}
 \nonumber\\&\quad\cdot \left(\frac{ q^{2(b-1)n}-q^{2(k-b-1)n}}{1 - q^{2nk}}+\frac{x^{k-b} q^{2(k-1-b)n}-x^{b} q^{2(b-1)n}}{1 - x^k q^{2kn}}\right).
\end{align}
Again, proceeding as in the proof of Theorem \ref{main1},
after applying \eqref{partial} and \eqref{genm2nt}, we obtain
\begin{align}
&\sum_{n \geq 0} \left(NT2(1,5,n)-NT2(4,5,n)\right) q^n
 \nonumber\\&=  \frac{(-q;q^2)_{\infty}}{(q^2;q^2)_{\infty}}\sum_{n \geq 1}
 \frac{ (-1)^{n}q^{2n^2+n} (q^{2n} - 1)^3 ( q^{4 n} - 1)}{(1 - q^{10n})^2} \{1 + 2q^{2n} + 4q^{4 n} +  2q^{6 n}  +  q^{8 n} \}\label{conm2mod51}
\intertext{and}
&\sum_{n \geq 0} \left(NT2(2,5,n)-NT2(3,5,n)\right) q^n
 \nonumber\\&= \frac{(-q;q^2)_{\infty}}{(q^2;q^2)_{\infty}}\sum_{n \geq 1}
 \frac{(-1)^{n} q^{2n^2+n} (q^{2n} - 1)^3 ( q^{4 n} - 1)}{(1 - q^{10n})^2} \{2q^{2n} + q^{4 n} +  2q^{6 n}   \}.\label{conm2mod52}
\end{align}
Equations \eqref{conm2mod51} and \eqref{conm2mod52} give
\begin{align}
&\sum_{n \geq 0} \left(NT2(1,5,n)-NT2(4,5,n)+2NT2(2,5,n)-2NT2(3,5,n)\right) q^n \nonumber \\
& =  \frac{(-q;q^2)_{\infty}}{(q^2;q^2)_{\infty}}\sum_{n \geq 1}
 \frac{(-1)^{n} q^{2n^2+n} (q^{2n} - 1)^3 ( q^{4 n} - 1)}{(1 - q^{10n})^2} \{1 + 6q^{2n} + 6q^{4 n} +  6q^{6 n}  +  q^{8 n}  \}
 \nonumber\\
 &\equiv  \frac{(-q;q^2)_{\infty}}{(q^2;q^2)_{\infty}}\sum_{n \geq 1}
 \frac{ (-1)^{n}q^{2n^2+n} (q^{2n} - 1)^2 ( 1-q^{4 n} )}{(1 - q^{10n})}\pmod{5}
  \nonumber\\&=   \frac{(-q;q^2)_{\infty}}{(q^2;q^2)_{\infty}}\sum_{n \geq 1}
 \frac{ (-1)^{n}q^{2n^2+n} (1-2q^{2n}+2q^{6n}-q^{8n} )}{1 - q^{10n}}
  \nonumber\\
&=   \frac{(-q;q^2)_{\infty}}{(q^2;q^2)_{\infty}}\sum_{{n \in \mathbb{Z} \atop n\neq0}}
 \frac{ (-1)^{n}q^{2n^2+n} (1-2q^{2n} )}{1 - q^{10n}}
   \nonumber\\
&=   \frac{(-q;q^2)_{\infty}}{(q^2;q^2)_{\infty}}\left(S_{2}(1)-2S_{2}(3)\right),\label{m2s131}
 \intertext{where}
 S_{2}(b):&=\sum_{{ n \in \mathbb{Z} \atop n\neq0}} \frac{(-1)^{n} q^{2n^{2}+b n}}{1-q^{10 n}}. \nonumber
\end{align}
By \cite[Eq. (5.7)]{Lo-Os1}, we have
\begin{equation}\label{m2s132}
\sum_{n \geq 0} \left\{N_{2}(1,5, n)-N_{2}(2,5, n)\right\} q^{n} \frac{(q^2;q^2)_{\infty}}{(-q;q^2)_{\infty}}=-S_{2}(1)+2S_{2}(3),
\end{equation}
where $N_{2}(b,k, n)$ denotes the number of partitions without repeated odd parts of $n$ whose $M_2$-rank is congruent to $b$ modulo $k$.
Equations \eqref{m2s131} and \eqref{m2s132} imply
\begin{align*}
&\sum_{n \geq 0} \left(NT2(1,5,n)- NT2(4,5,n)+ 2NT2(2,5,n)-2 NT2(3,5,n)\right) q^n
\nonumber\\&\equiv\sum_{n \geq 0} \left\{N_{2}(2,5, n)- N_{2}(1,5, n)\right\} q^{n} \pmod{5},
\end{align*}
which together with \cite[Eq. (1.7)]{Lo-Os1} gives (\ref{coneqm25}).
\end{proof}

\section{Concluding Remarks}

There are several directions for future study. First, Dyson also conjectured in \cite{dyson} the existence of a partition statistic called the crank which would combinatorially explain Ramanujan's congruences for the partition function modulo $5$, $7$ and $11$. In \cite{andgar}, this statistic was defined and Dyson's conjecture was proven. The crank of a partition is either the largest part, if 1 does not occur, or the difference between the number of parts larger than the number of $1$'s and the number of $1$'s, if $1$ does occur. Let $M_{\omega}(m,k,n)$ denote the number of ones in the partitions of $n$ with crank congruent to $m$ modulo $k$. It appears that there are further congruences and relations for $NT(m,k,n)$ and $M_{\omega}(m,k,n)$. For example, using the techniques from \cite{andrews}, one can prove that for $i=1$, $3$, $4$, $5$,

\begin{align}
&NT(1,7,7n+i)-NT(6,7,7n+i)+2NT(3,7,7n+i)-2NT(4,7,7n+i) \equiv0\pmod{7} \label{newNT1} \\
\intertext{and for $i=0,1,5$,} \nonumber \\
&NT(2,7,7n+i)-NT(5,7,7n+i)+4NT(3,7,7n+i)-4NT(4,7,7n+i) \equiv0\pmod{7}. \label{newNT2} 
\end{align}
The details are left to the interested reader. We note that for $i=1$, $5$, (\ref{newNT1}) and (\ref{newNT2}) follow from \cite[Theorem 2]{andrews} and \cite[Corollary 1.4]{chern1}. In addition, we make the following

\begin{conjecture} For all $n \in \mathbb{N}$, we have
\begin{align}
&\sum_{n \geq 0}(NT(1,7,7n+5)-NT(6,7,7n+5)+3NT(2,7,7n+5)-3NT(5,7,7n+5))q^n\nonumber \\
&\qquad \qquad \qquad \qquad \qquad \qquad \qquad \qquad \qquad \qquad \qquad \qquad \qquad = \frac{-7(q^7;q^7)_{\infty}^3(q^3,q^4;q^7)_{\infty}}{(q,q^6;q^7)_{\infty} (q^2, q^5; q^7)^2_{\infty}},  \label{id7125}\\
&\sum_{n \geq 0}(NT(1,7,7n+4)-NT(6,7,7n+4)+2NT(3,7,7n+4)-2NT(4,7,7n+4))q^n\nonumber \\
&\qquad \qquad \qquad \qquad \qquad \qquad \qquad \qquad \qquad \qquad \qquad \qquad \qquad =\frac{-7(q^7;q^7)_{\infty}^3(q^3,q^4;q^7)^2_{\infty}}{(q,q^6;q^7)_{\infty} (q^2, q^5; q^7)^3_{\infty}},  \label{id7135} \\
& {} \nonumber \\
& NT(1,11,11n+6)-NT(10,11,11n+6) + 3NT(2,11,11n+6)-3NT(9,11,11n+6)\nonumber\\
&-4NT(3,11,11n+6)+4NT(8,11,11n+6) +3NT(4,11,11n+6)-3NT(7,11,11n+6) \nonumber\\
& +3NT(5,11,11n+6)-3NT(6,11,11n+6)\equiv0\pmod{11},\label{con116}\\
& {} \nonumber \\
&NT(1,11,11n+1)-NT(10,11,11n+1)- 3NT(2,11,11n+1)+3NT(9,11,11n+1)\nonumber\\
&+5NT(3,11,11n+1)-5NT(8,11,11n+1) -2NT(4,11,11n+1)+2NT(7,11,11n+1)\nonumber\\
&+4NT(5,11,11n+1)-4NT(6,11,11n+1)\equiv0\pmod{11}, \label{con111} \\
& {} \nonumber \\
&NT(1,13,13n+1)-NT(12,13,13n+1)+NT(2,13,13n+1)-NT(11,13,13n+1)\nonumber\\
&+6NT(3,13,13n+1)-6NT(10,13,13n+1) +3NT(6,13,13n+1)-3NT(7,13,13n+1)\nonumber\\
&\equiv0\pmod{13}, \label{con131} \\
& {} \nonumber \\
&NT(1,13,13n+3)-NT(12,13,13n+3)+3NT(3,13,13n+3)-3NT(10,13,13n+3)\nonumber\\
&-4NT(4,13,13n+3)+4NT(9,13,13n+3) +NT(5,13,13n+3)-NT(8,13,13n+3)\nonumber\\
&-2NT(6,13,13n+3)+2NT(7,13,13n+3)\equiv0\pmod{13}, \label{con133} \\
& {} \nonumber \\
&M_{\omega}(1,5,5n+4)-M_{\omega}(4,5,5n+4)=2M_{\omega}(3,5,5n+4)-2M_{\omega}(2,5,5n+4), \label{eqc54}\\
& {} \nonumber \\
&M_{\omega}(1,5,5n+i)-M_{\omega}(4,5,5n+i)+2NT(2,5,5n+i)-2NT(3,5,5n+i)\equiv0\pmod{5} \label{concr504} \\
\intertext{for $i=0$, $4$,} \nonumber
& {} \nonumber \\
&M_{\omega}(1,5,5n+2)-M_{\omega}(4,5,5n+2)=2NT(3,5,5n+2)-2NT(2,5,5n+2), \label{eqcr52}\\
& {} \nonumber \\
&NT(1,5,5n+i)-NT(4,5,5n+i) + 2M_{\omega}(2,5,5n+i)-2M_{\omega}(3,5,5n+i)\equiv0\pmod{5} \label{conrc512}\\
\intertext{for $i=1$, $2$,} \nonumber
& {} \nonumber \\
& \sum_{n \geq 0}\big(NT(1,5,5n+4)-NT(4,5,5n+4) + 2M_{\omega}(2,5,5n+4)-2M_{\omega}(3,5,5n+4)\big) q^n \nonumber \\
&\qquad \qquad \qquad \qquad \qquad \qquad \qquad \qquad \qquad \qquad \qquad \qquad \qquad \quad = \frac{-5 (q^5;q^5)_{\infty}^4}{(q)_{\infty} }, \label{idcr54}\\
& {} \nonumber \\
&M_{\omega}(1,5,5n+4)-M_{\omega}(4,5,5n+4)=4NT(4,5,5n+4)-4NT(1,5,5n+4), \label{eqcr54} \\
& {} \nonumber \\
&M_{\omega}(1,7,7n+i)-M_{\omega}(6,7,7n+i)+2M_{\omega}(3,7,7n+i)-2M_{\omega}(4,7,7n+i) \equiv0\pmod{7} \label{concr713}\\
\intertext{for $i=0$, $2$, $5$, $6$,} \nonumber
& {} \nonumber \\
&M_{\omega}(2,7,7n+i)-M_{\omega}(5,7,7n+i)-3M_{\omega}(3,7,7n+i)+3M_{\omega}(4,7,7n+i) \equiv0\pmod{7} \label{concr723} \\
\intertext{for $i=0$, $1$, $4$, $5$.} \nonumber
\end{align}
\end{conjecture}

Here, \eqref{eqc54} implies \cite[Corollary 1.5]{chern1}. Second, do similar congruences and/or relations exist for total number of parts functions associated to other partitions statistics, for example, ranks of Durfee symbols \cite{and}, the first and second residual crank for overpartitions \cite{blo} and $d$th residual crank for overpartitions \cite{ams}, the $k$-rank \cite{gar1} or the $M_d$-rank for overpartitions \cite{cjshs}, \cite{morrill}? Finally, a mock modular perspective (such as in \cite{gar2}, \cite{cjs1}, \cite{cjs2} or \cite{mao}) which explains the occurrences of (\ref{ntmod5}), (\ref{ntmod7}), (\ref{coneqm2o5})--(\ref{coneqm25}) and (\ref{id7125})--(\ref{concr723}) would be most welcome.

\section*{Acknowledgements}
The second author was partially supported by the National Natural Science Foundation of China (Grant No. 12071331 and No. 11971341).


\begin{thebibliography}{999}

\bibitem{ams}
A. Al-Saedi, T. Morrill and H. Swisher, \emph{Inequalities for the $d$th residual crank moments of overpartitions}, Int. J. Number Theory \textbf{16} (2020), no. 7, 1599--1606.

\bibitem{and}
G.E. Andrews, \emph{Partitions, Durfee symbols, and the Atkin-Garvan moments}, Invent. Math. \textbf{169} (2007), no. 1, 37--73.

\bibitem{andrews}
G.E. Andrews, \emph{The Ramanujan-Dyson identities and George Beck's congruence conjectures}, Int. J. Number Theory, to appear.

\bibitem{andgar}
G.E. Andrews, F.G. Garvan, \emph{Dyson's crank of a partition}, Bull. Amer. Math. Soc. (N.S.) \textbf{18} (1988), no. 2, 167--171.

\bibitem{asd}
A.O.L. Atkin, P. Swinnerton-Dyer, \emph{Some properties of partitions}, Proc. London Math. Soc. (3) \textbf{4} (1954), 84--106.

\bibitem{bg}
A. Berkovich, F. G. Garvan, \emph{Some observations on Dyson's new symmetries of partitions}, J. Combin. Theory, Ser. A \textbf{100} (2002), 61--93.

\bibitem{bl}
K. Bringmann, J. Lovejoy, \emph{Rank and congruences for overpartition pairs}, Int. J. Number Theory \textbf{4} (2008), no. 2, 303--322.

\bibitem{blo}
K. Bringmann, J. Lovejoy and R. Osburn, \emph{Rank and crank moments for overpartitions}, J. Number Theory \text{129} (2009), no. 7, 1758--1772.

\bibitem{gen}
S. H. Chan, \emph{Generalized Lambert series identities}, Proc. London Math. Soc. (3) \textbf{91} (2005), no. 3, 598--622.

\bibitem{ramphi}
S. H. Chan, \emph{Congruences for Ramanujan's $\phi$ function}, Acta Arith. \textbf{153} (2012), no. 2, 161--189.

\bibitem{chern1}
S. Chern, \emph{Weighted partition rank and crank moments. I. Andrews-Beck type congruences}, preprint.

\bibitem{dyson}
F.J. Dyson, \emph{Some guesses in the theory of partitions}, Eureka \textbf{8} (1944), 10--15.

\bibitem{gar1}
F.G. Garvan, \emph{Generalizations of Dyson's rank and non-Rogers-Ramanujan partitions}, Manuscripta Math. \textbf{84} (1994), no. 3-4, 343--359.

\bibitem{gar2}
F.G. Garvan, \emph{Transformation properties for Dyson's rank function}, Trans. Amer. Math. Soc. \textbf{371} (2019), no. 1, 199--248.

\bibitem{gr}
G. Gasper, M. Rahman, \emph{Basic hypergeometric series}, Second edition. Encyclopedia of Mathematics and its Applications, 96. Cambridge University Press, Cambridge, 2004.

\bibitem{cjs1}
C. Jennings-Shaffer, \emph{Overpartition rank differences modulo $7$ by Maass forms}, J. Number Theory \textbf{163} (2016), 331--358.

\bibitem{cjs2}
C. Jennings-Shaffer, \emph{The generating function of the $M_2$-rank of partitions without repeated odd parts as a mock modular form}, Trans. Amer. Math. Soc. \textbf{371} (2019), no. 1, 249--277.

\bibitem{cjshs}
C. Jennings-Shaffer, H. Swisher, \emph{Mock modularity of the $M_{d}$-rank of overpartitions}, J. Math. Anal. Appl. \textbf{466} (2018), no. 2, 1144--1189.

\bibitem{Lo1}
J. Lovejoy, \emph{Rank and conjugation for the Frobenius representation of an overpartition}, Ann. Comb. \textbf{9} (2005), 321--334

\bibitem{Lo2}
J. Lovejoy, \emph{Rank and conjugation for a second Frobenius representation
of an overpartition}, Ann. Comb. \textbf{12} (2008), 101--113.

\bibitem{Lo-Os1}
J. Lovejoy, R. Osburn, \emph{$M_2$-rank differences for partitions without repeated odd parts}, J. Th\'eor. Nombres Bordeaux \textbf{21} (2009), no. 2, 313--334.

\bibitem{Lo-Os2}
J. Lovejoy, R. Osburn, \emph{$M_2$-rank differences for overpartitions}, Acta. Arith. \textbf{144} (2010), no. 2, 193--212.

\bibitem{mao}
R. Mao, \emph{$M_2$-rank of overpartitions and harmonic weak Maass forms}, J. Math. Anal. Appl. \textbf{426} (2015), no. 2, 794--804.

\bibitem{morrill}
T. Morrill, \emph{Two families of buffered Frobenius representations of overpartitions}, Ann. Comb. \textbf{23} (2019), no. 1, 103--141.

\end{thebibliography}
\end{document}